\documentclass[11pt]{article}

\usepackage{tikz}
\usepackage{pgfplots}

\usetikzlibrary{patterns,arrows,decorations.pathreplacing}
\usetikzlibrary{calc}

\usepackage{amsmath}
\usepackage{amsfonts}
\usepackage{amssymb}
\usepackage{amsthm}
\usepackage{breqn}
\usepackage{setspace}
\usepackage{fullpage}
\usepackage{enumitem}
\usepackage{bbold} 
\usepackage{comment}
\usepackage{hyperref}
\usepackage{breqn}

\newtheorem{theorem}{Theorem}

\newtheorem{remark}{Remark}

\newtheorem{conjecture}{Conjecture}

\newcommand{\abs}[1]{\left\lvert{#1}\right\rvert}
\newcommand{\floor}[1]{\left\lfloor{#1}\right\rfloor}
\newcommand{\ceil}[1]{\left\lceil{#1}\right\rceil}

\DeclareMathOperator{\ex}{ex}

\title{An Erd\H{o}s-Gallai type theorem for vertex colored graphs}

\author{
Nika Salia\thanks{Central European University, Budapest. email: \texttt{Nika\_Salia@phd.ceu.edu} } \and 
Casey Tompkins\thanks{Alfr\'ed R\'enyi Institute of Mathematics, Hungarian Academy of Sciences. email: \texttt{ctompkins496@gmail.com}} \and 
Oscar Zamora\thanks{Central European University, Budapest. email: \texttt{oscarz93@yahoo.es}} \thanks{Universidad de Costa Rica, San Jos\'e. }
}

\begin{document}
\maketitle
\begin{abstract}
While investigating odd-cycle free hypergraphs, Gy\H{o}ri and Lemons introduced a colored version of the classical theorem of Erd\H{o}s and Gallai on $P_k$-free graphs.  They proved that any graph $G$ with a proper vertex coloring and no path of length $2k+1$ with endpoints of different colors has at most $2kn$ edges.  We show that  Erd\H{o}s and Gallai's original sharp upper bound of $kn$ holds for their problem as well.  We also introduce a version of this problem for trees and present a generalization of the Erd\H{o}s-S\'os conjecture.
\end{abstract}

We denote by $P_\ell$ the path of length $\ell$ (that is, containing $\ell$ edges).  For a graph $G$, we denote by $E(G)$ and $V(G)$ the edge and vertex set of $G$, respectively.   We begin by recalling the theorems of Erd\H{o}s and Gallai about graphs without long paths and cycles.  
\begin{theorem}[Erd\H{o}s-Gallai \cite{erdHos1959maximal}] \label{EGpath}  Let $G$ be an $n$-vertex graph with no $P_\ell$, then 
\begin{displaymath}
\abs{E(G)} \le \frac{\ell-1}{2}n,
\end{displaymath}
and equality holds if and only if $\ell$ divides $n$ and $G$ is the disjoint union of $\frac{n}{\ell}$ cliques of size $\ell$.
\end{theorem}
\begin{theorem}[Erd\H{o}s-Gallai \cite{erdHos1959maximal}]  \label{EGcycle} Let $G$ be an $n$-vertex graph with no $C_m$ for all $m \ge \ell$, then 
\begin{displaymath}
\abs{E(G)} \le \frac{(\ell-1)(n-1)}{2}.
\end{displaymath}
and equality holds if and only if $\ell-2$ divides $n-1$ and $G$ is the union of $\frac{n-1}{\ell-2}$ disjoint cliques of size $\ell-1$ sharing a fixed vertex.
\end{theorem}
In fact, Theorem \ref{EGpath} was deduced as a simple corollary of Theorem \ref{EGcycle}.  In a more recent paper, Gy\H{o}ri and Lemons \cite{gyHori20123} investigated the extremal number of hypergraphs avoiding so-called Berge-cycles. To this end, they introduced a generalization of the theorem of Erd\H{o}s and Gallai about paths.   By a proper vertex coloring of a graph $G$, we mean a coloring of the vertices of $G$ such that no two adjacent vertices are the same color. Gy\H{o}ri and Lemons proved the following.

\begin{theorem}[Gy\H{o}ri-Lemons \ref{EGcycle}]
\label{main}
Let $G$ be an $n$-vertex graph with a proper vertex coloring such that $G$ contains no $P_{2k+1}$ with endpoints of different colors,  then
\begin{displaymath}
\abs{E(G)} \le2 k n.  
\end{displaymath}
\end{theorem}
We show that the factor of $2$ in Theorem \ref{main} is not needed and, thus, recover the original upper bound from the Erd\H{o}s-Gallai theorem.  We also determine which graphs achieve this upper bound. 
\begin{theorem}
\label{main2}
Let $G$ be an $n$-vertex graph with a proper vertex coloring such that $G$ contains no $P_{2k+1}$ with endpoints of different colors,  then
\begin{displaymath}
\abs{E(G)} \le k n, 
\end{displaymath}
and equality holds if and only if $2k+1$ divides $n$ and $G$ is the union of $\frac{n}{2k+1}$ disjoint cliques of size $2k+1$.
\end{theorem}

\begin{proof}[Proof of Theorem \ref{main2}]
By induction on the number of vertices, we may assume that $G$ is connected and has minimum degree $\delta(G) \ge k$. Indeed, if $\delta(v)<k$ then 
\begin{displaymath}
e(G)=e(G-v)+\delta(v) \le k(n-1)+k-1 < kn.
\end{displaymath}
If $G$ is $C_{\ell}$-free for all $\ell \ge 2k+1$, then by Theorem \ref{EGcycle} we have
\begin{displaymath}
\abs{E(G)} \le \frac{(n-1)2k}{2} < kn.  
\end{displaymath}
Thus, assume there is a cycle of length at least $2k+1$, and let $C$ be the smallest such cycle with length $\ell$. Let the vertices of $C$  be $v_0,v_1,v_2,\dots,v_{\ell-1},v_0$, consecutively. Addition and subtraction in subscripts will always be taken modulo $\ell$.  We say that an edge $e$ is \emph{outgoing} if it has one vertex in $V(C)$ and the other in $V(G) \setminus V(C)$. We say a vertex $v \in V(C)$ is outgoing if it is contained in an outgoing edge.

\noindent\begin{bf} Case 1. \end{bf}    Suppose $\ell \ge 2k+4$.  Since we have chosen $\ell$ to be the length of the smallest $C_\ell$ with $\ell \ge 2k+1$, we have that $v_0$ cannot be adjacent to any of $v_2,v_3,\dots,v_{\ell-2k}$ nor any of $v_{2k},v_{2k+1},\dots,v_{\ell-2}$, for otherwise we would have a shorter cycle of length at least $2k+1$.  Also note that $v_0$ is adjacent to $v_1$ and $v_{\ell-1}$.  

Observe that $v_0$ cannot have two consecutive neighbors in the $\ell$-cycle.  Indeed, if $v_i$ and $v_{i+1}$ are neighbors of $v_0$, then we have the following  $(2k+1)$-paths starting at $v_1$: $v_1,v_2,\dots,v_{2k+1},v_{2k+2}$ and $v_1,v_2,\dots,v_i,v_0,v_{i+1},v_{i+2},\dots,v_{2k},v_{2k+1}$.  Thus, $v_{2k+1}$ and $v_{2k+2}$ would have to be the same color, but this is impossible since they are neighbors.

If $v_0$ has a neighbor outside of $C$, say $u_0$, then we have two paths of length $2k+1$: $u_0,v_0,v_1,\dots,v_{2k}$ and $v_{2k},v_{2k-1},\dots,v_0,v_{\ell-1}$. It follows that $u_0$ and $v_{\ell-1}$ have the same color. Similarly, $u_0$ and $v_1$ have the same color. Thus, $v_{\ell-1}$ and $v_1$ also have the same color, and similarly, for every $i$ such that $v_i$ is outgoing, we can conclude $v_{i-1}$ and $v_{i+1}$ have the same color.

If $\ell=2k+4$ and there is an outgoing vertex, say $v_0$, then $v_1$ and $v_{2k+3}$ have the same color (from the previous paragraph), a contradiction since $v_1$ and $v_{2k+2}$ also have the same color (they are endpoints of a length $2k+1$ path along the cycle $C$). 
If there is no outgoing vertex in $V(C)$, then $C$ uses all vertices of the graph. Since no vertex of the cycle has two consecutive neighbors, it follows that each degree is bounded by $2+\lceil \frac{2k-5}{2}\rceil \leq k$ and so the number of edges is at most $\frac{(2k+4)k}{2} = \frac{nk}{2}< nk.$

If $\ell \geq 4k+5$, we will show that $v_0$ has an outgoing edge from the $\ell$-cycle $C$.  Suppose not, then since $v_0$ does not have consecutive neighbors, it follows that $v_0$ has at most
\begin{displaymath}
2 + \ceil{ \frac{2k- (\ell - 2k+1)}{2} } \le k-1
\end{displaymath}
neighbors, a contradiction.  Thus, $v_0$ and similarly every other $v_i$ has an outgoing neighbor, and it follows that for every $i$, $v_i$ and $v_{i+2}$ have the same color. Hence $v_0$ and $v_{2k}$ have the same color, contradicting that $v_0$ and $v_{2k+1}$ have the same color, since they are endpoints of a $P_{2k+1}$.

\noindent\begin{bf} Case 2. \end{bf} Suppose $\ell = 2k+3$. For all $0 \le i \le \ell-1$, $v_{i+2},v_{i+1},\dots,v_{\ell-1},v_0,\dots,v_i$ is a path of length $2\ell+1$, and so $v_i$ and $v_{i+2}$ have the same color.  Thus, $v_0$ and $v_{2k+2}$ have the same color, but they are adjacent, contradiction.

\noindent\begin{bf} Case 3. \end{bf} Suppose that $\ell = 2k+2$.  This is impossible since $v_0,v_1,\dots,v_{2k+1}$ is a path of length $2k+1$ but $v_0$ and $v_{2k+1}$ are adjacent, contradiction. 

\noindent\begin{bf} Case 4. \end{bf} Finally, suppose $\ell = 2k+1$.    If no edge is outgoing, then we are done, since by connectivity the total number of edges in the graph is at most $\binom{2k+1}{2} = k n$. If indeed the total number of edges is $kn$, then $G$ is a clique.  This is the only case when equality holds. From here on, we will assume there is an outgoing edge.  


Observe that if $u$ is not a vertex of $C$, then $u$ cannot have two consecutive neighbors in $C$, for otherwise we would have a  cycle of length $2k+2$, and we are done as in Case 3.   Moreover,  $u$ cannot be connected to $v_i$ and $v_{i+3}$, since there would be two paths of length $2k+1$ from $u$ to $v_{i+1}$ and $v_{i+2}$. It follows that $u$ can have at most $k-1$ neighbors in $C$ and, thus, must have a neighbor outside $C$.

If there are some two consecutive outgoing vertices in $C$, then we may take two such vertices $v_i$ and $v_{i+1}$ so that the next vertex $v_{i+2}$ is outgoing.  Suppose $\{v_{i+2},u_{i+2}\}$ is an outgoing edge. By the previous observation, there is an edge $\{u_{i+2},w_{i+2}\}$ where $w_{i+2}$ does not belong to the cycle.  It is easy to see that $v_{i}$ and also $v_{i+1}$ cannot have two consecutive neighbors from $C$, and so each has degree $k$. By removing these two vertices, we remove $2k-1$ edges, and by the induction hypothesis the resulting graph has at most $k(n-2)$ edges. So $e(G)<kn$.

For every $i$, either $v_{i+1}$ or $v_{i+2}$ is an outgoing vertex. Hence the vertex $v_{i}$ has either the same color  as $v_{i+2}$, if $v_{i+1}$ is an outgoing vertex, or the same color as $v_{i+4}$, if $v_{i+2}$ is an outgoing vertex. Hence by repeatedly applying this argument we obtain that $v_0$ has the same color as $v_{2k}$ or $v_1$, contradiction. \qedhere



\end{proof}

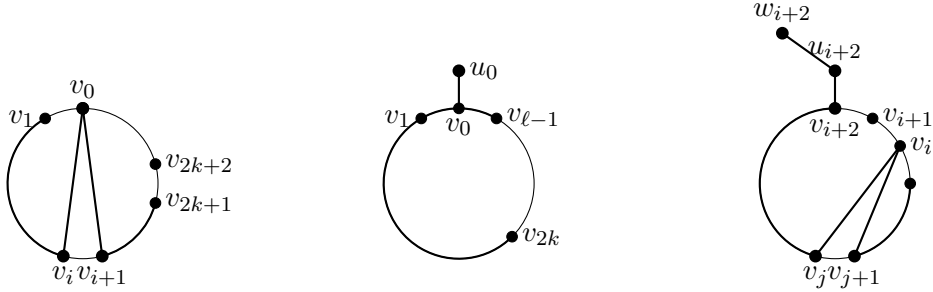
\begin{figure}
\begin{center}

\begin{tikzpicture}

\filldraw[thick] (0,1) circle (2pt) node[align=center,above]{$v_0$}  -- (xyz polar cs:angle=255,radius=1) node[align=center,below]{$v_i$} circle (2pt);
\filldraw[thick] (0,1) circle (2pt)   -- (xyz polar cs:angle=285,radius=1) node[align=center,below]{$v_{i+1}$} circle (2pt);
\filldraw (xyz polar cs:angle=120,radius=1) circle(2pt)node[left]{$v_1$};

 \filldraw  (xyz polar cs:angle=15,radius=1)node[right]{$v_{2k+2}$} circle (2pt);
\draw (1,0) arc (0:360:1cm and 1cm);

\draw[thick] (xyz polar cs:angle=120,radius=1) arc (120:255: 1cm and 1cm); 

\draw[thick] (xyz polar cs:angle=285,radius=1) arc (285:345: 1cm and 1cm);

\filldraw (xyz polar cs:angle=345,radius=1) circle (2pt) node[right]{$v_{2k+1}$} ;


\filldraw (5,1) circle (2pt) node[align=center,below]{$v_0$}  ;
\filldraw[thick] (5,1.5) node[right]{$u_0$}circle (2pt) -- (5,1);

\filldraw  (5,0)+ (xyz polar cs:angle=120,radius=1) circle(2pt)node[left]{$v_1$};

 \filldraw  (5,0)+ (xyz polar cs:angle=60,radius=1)node[right]{$v_{\ell-1}$} circle (2pt);
\draw (6,0) arc (0:360:1cm and 1cm);

\draw[thick] (5,0)+(xyz polar cs:angle=60,radius=1) arc (60:315: 1cm and 1cm);

\filldraw (5,0)+ (xyz polar cs:angle=315,radius=1) circle (2pt) node[right]{$v_{2k}$} ;

\begin{scope}[shift={(10,0)}]

\draw (1,0) arc(0:360: 1cm and 1cm);

\filldraw[thick] (0,1) circle (2pt) node[align=center,below]{$v_{i+2}$} -- (0,1.5) circle (2pt)node[above]{$u_{i+2}$};

\filldraw[thick] (0,1.5) -- (-.7,2) node[align=center,above]{$w_{i+2}$} circle (2pt);
  
\filldraw (xyz polar cs:angle=60,radius=1) circle(2pt)node[right]{$v_{i+1}$};
\filldraw (xyz polar cs:angle=30,radius=1) circle(2pt)node[right]{$v_i$};

\draw[thick] (xyz polar cs:angle=90,radius=1) arc (90:255: 1cm and 1cm); 

\draw[thick] (xyz polar cs:angle=285,radius=1) arc (285:360: 1cm and 1cm) ; 

\filldraw (xyz polar cs:angle=360,radius=1) circle (2pt); 

\filldraw[thick] (xyz polar cs:angle=30,radius=1) -- (xyz polar cs:angle=255,radius=1) node[align=center,below]{$v_j$} circle (2pt);
\filldraw[thick] (xyz polar cs:angle=30,radius=1) -- (xyz polar cs:angle=285,radius=1) node[align=center,below]{$v_{j+1}$} circle (2pt);

\end{scope}
\end{tikzpicture}

\caption{The picture on the left shows Case~1, and the other pictures show Case~4.}
\end{center}
\end{figure}


We believe that an analogue of Theorem \ref{main} should hold in the setting of trees.  Recall that the extremal number $\ex(n,H)$ of a graph $H$ is defined to be the largest number of edges an $n$-vertex graph may have if it does not contain $H$ as a subgraph.  Erd\H{o}s and S\'os made the following famous conjecture about the extremal number of trees.

\begin{conjecture}[Erd\H{o}s-S\'os \cite{erd1964extremal}]
\label{es}
Let $T$ be a tree with $k$ edges, then $\ex(n,T) \le \frac{(k-1)n}{2}$.
\end{conjecture}
A proof of Conjecture \ref{es} for sufficiently large trees has been announced by Ajtai,  Koml\'os, Simonovits and Szemer\'edi \cite{erdossosconj}.

We introduce a new variation of the extremal function $\ex(n,T)$ in the case of trees.  Let $\ex^c(n,T)$ denote the maximimum number of edges possible in an $n$-vertex graph $G$ with a proper vertex coloring (using any number of edges), such that in every copy of $T$ in $G$ the leaves of $T$ are all the same color.

\begin{theorem}
Let $T$ be a tree with $k$ edges such that in the (unique) proper vertex $2$-coloring of $T$ all leaves are not the same color,  then $\ex^c(n,T) \le (k-1)n$.
\end{theorem}
\begin{proof}
 There is a path of odd length in $T$ with endpoints which are leaves. Let $G$ be an $n$-vertex graph with more than $(k-1)n$ edges with a proper vertex coloring.  We may find a subgraph $G'$ of $G$ with average degree at least that of $G$ and minimum degree greater than $k-1$.  The proper coloring of  $G$ induces a proper coloring of $G'$ and so applying Theorem \ref{main} we may find a copy of $P_{2 \ell+1}$ in $G'$ with endpoints of distinct colors. We may now build up the rest of the tree in a greedy fashion as every degree in $G'$ is at least $k$ and $T$ has $k+1$ vertices. Thus, we have found a copy of $T$ in the graph $G$ with leaves of at least two colors.
\end{proof}

\begin{theorem}
\label{trees}
Let $T$ be a tree with $k$ edges such that in the proper vertex $2$-coloring of $T$ all leaves are the same color, then $\ex^c(n,T) = \floor{\frac{n^2}{4}}$, provided $n$ is sufficiently large.
\end{theorem}

\begin{proof}
The fact that all leaves are colored the same by a $2$-coloring implies that all paths between a pair of leaves have even length.  We add an edge $e$ to $T$ connecting an arbitrary pair of leaves, and let $G$ be the resulting graph.  Since $G$ has an odd cycle, its chromatic number is clearly 3, and the deletion of $e$ yields a 2-chromatic graph (so $G$ is edge-critical). It follows from a theorem of Simonovits \cite{simonovits1968method} that if $n$ is sufficiently large, the extremal number of $G$ is precisely $\ex(n,G) =  \floor{\frac{n^2}{4}}$.  Thus, in any $n$-vertex graph with more than $\floor{\frac{n^2}{4}}$ edges we have a copy $T$ with two adjacent leaves, and so in any proper coloring of this graph we have a copy of $T$ with leaves of at least 2 colors.  It follows that $\ex^c(n,T) \le \floor{\frac{n^2}{4}}$, and this bound is realized by the complete bipartite graph $K_{\floor{\frac{n}{2}},\ceil{\frac{n}{2}}}$.
\end{proof}

\begin{remark}
The paths of even length $P_{2k}$ are a special case of Theorem  \ref{trees}. Here better bounds on $n$ are known to exist.  For example, the result of F\"uredi  \cite{furedi2015extremal} on the extremal number of odd cycles implies that $n \ge 4k$ is sufficient.
\end{remark}

We believe that a strengthening of Conjecture \ref{es} should hold for trees whose 2-coloring yields two leaves of different colors.  

\begin{conjecture}
\label{es2}
Let $T$ be a tree with $k$ edges such that in the proper vertex $2$-coloring of $T$ all leaves are not the same color,  then $\ex^c(n,T) \le \frac{(k-1)n}{2}$.
\end{conjecture}

One would hope that Conjecture~\ref{es2} could be deduced directly from Conjecture~\ref{es}, but unfortunately this does not seem to be the case. We take a first step towards Conjecture~\ref{es2} by proving it in the case of double stars.

\begin{figure}
\begin{center}
\begin{tikzpicture}

	\draw  (1,0) circle  (2pt)  node[align=center, above]{$u$} -- (2,0) circle  (2pt)  node[align=center, above]{$v$};
\filldraw (1,0) -- (0,1) circle  (2pt);
\filldraw (1,0) -- (0,.5) circle  (2pt);
\draw (0,0)node[align=center]{$\vdots$};
\filldraw (1,0) -- (0,-0.5) circle  (2pt);
\draw (0.5,.25) arc (0:360:0.5cm and 1cm) ;

\draw (0,1.2) node[align=center, above]  {$A$};
\filldraw (2,0) -- (3,1) circle  (2pt);
\filldraw (2,0) -- (3,.5) circle  (2pt);
\filldraw (2,0) -- (3,-0.5) circle  (2pt);

\draw (3,0)node[align=center]{$\vdots$};
\draw (3,1.2) node[align=center, above]  {$B$};
\draw (3.5,0.25) arc (0:360:0.5cm and 1cm);

\begin{scope}[shift={(4,0)}]
	\draw  (1,0) circle  (2pt)  node[align=center, above]{$u$} -- (2,0) circle  (2pt)  node[align=center, above]{$v$};
\filldraw (2,0) -- (3,1) circle  (2pt);
\filldraw (2,0) -- (3,.65) circle  (2pt);
\filldraw (2,0) -- (3,-0.5) circle  (2pt);
\draw (3,0)node[align=center]{$\vdots$};
\draw (3,1.2) node[align=center, above]  {$B$};
\draw (3.5,0.25) arc (0:360:0.5cm and 1cm);
\filldraw  (2,0) -- (3,0.2) node[align=center, above]{$w$} circle (2pt);

\draw (3,0.2) -- (4.5,1) circle  (2pt);
\draw (3,0.2) -- (4.5,.65) circle  (2pt);
\draw (3,0.2) -- (4.5,0.2) circle  (2pt);
\draw (3,0.2) -- (4.5,-0.5) circle  (2pt);
\draw (4.5,0)node[align=center]{$\vdots$};
\draw (4.5,1.2) node[align=center, above]  {$C$};
\draw (5,0.25) arc (0:360:0.5cm and 1cm);

\end{scope}

\begin{scope}[shift={(10,0)}]

	\draw  (1,0) circle  (2pt)  node[align=center, above]{$u$} -- (3,0.2) node[align=center, above]{$x$}  ;
\filldraw (1,0) -- (3,1) circle  (2pt);
\filldraw (1,0) -- (3,.65) circle  (2pt);
\filldraw (1,0) -- (3,-0.5) circle  (2pt);
\draw (3,0)node[align=center]{$\vdots$};
\draw (3,1.2) node[align=center, above]  {$B$};
\draw (3.5,0.25) arc (0:360:0.5cm and 1cm);

\filldraw (3,0.2) circle (2pt);

\draw (3,0.2) -- (4.5,1) circle  (2pt);
\draw (3,0.2) -- (4.5,.65) circle  (2pt);
\draw (3,0.2) -- (4.5,0.2) circle  (2pt);
\draw (3,0.2) -- (4.5,-0.5) circle  (2pt);
\draw (4.5,0)node[align=center]{$\vdots$};
\draw (4.5,1.2) node[align=center, above]  {$C$};
\draw (5,0.25) arc (0:360:0.5cm and 1cm);

\filldraw (1,0) -- (0,1) node[align=center, above]{$y$}   circle  (2pt);

\end{scope}

\end{tikzpicture}

\end{center}
\caption{Theorem \ref{ds} and its proof.}
\label{doublestar}
\end{figure}
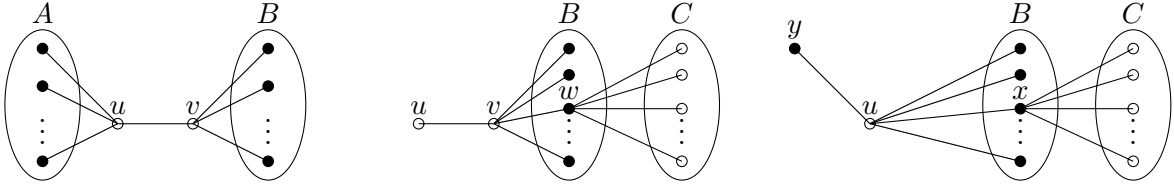

\begin{theorem}
\label{ds}
For positive integers $a$ and $b$, let $S_{a,b}$ denote the tree on $a+b+2$ vertices consisting of an edge $\{u,v\}$ where $\abs{N(u)\setminus v}=a$, $\abs{N(v)\setminus u}=b$ and $N(u) \cap N(v) = \varnothing$ (See Figure \ref{doublestar}, left).  We have $\ex^c(n,S_{a,b}) \le \frac{a+b}{2}n$.
\end{theorem}
\begin{proof}
Let $G$ be a vertex colored graph with $\abs{E(G)} > \abs{V(G)}\frac{a+b}{2}$. Without loss of generality, suppose $a\leq b$. We may assume by induction that $\delta(G)>\frac{a+b}{2} \geq a$. Since $\ex(m,S_{a,b}) = m\frac{a+b}{2}$ (see, for example \cite{mclennan2005erdHos}), it follows that  $G$ contains a copy of $S_{a,b}$.  Suppose this copy is defined by the edge $\{u,v\}$ together with the disjoint sets $A\subseteq N(u)$, $B\subseteq N(v)$ with $|A|=a, |B|=b$. Now, if there is more than one color in $A\cup B$, then we are done. So suppose the color of all vertices in $A \cup B$ is the same. Hence $A \cup B$ is an independent set.

If $u$ is not adjacent to some $w \in B$ (See Figure \ref{doublestar}, middle),  since $\abs{N(w)} \geq a +1$,  we can pick $C \subseteq N(w)\setminus\{u,v\}$ of size $a$. So the edge $\{v,w\}$ together with the sets $B' = (B\cup\{u\})\backslash\{w\}$ and $C$  define a $S_{a,b}$ where the colors of all vertices in $C$ are different from the colors of $B'\backslash\{u\}$. 

If $u$ is adjacent to all $w \in B$, then fix $x\in B$ (See Figure \ref{doublestar}, right). Since $\abs{N(x)} \geq a +1$, we can pick $C \subseteq N(x)\setminus\{u\}$ of size $a$. Let $y \in A$ and define $B' = (B\cup\{y\})\backslash\{x\}$.  Observe that $B'\subseteq N(u)$, and  the edge $\{u,x\}$ together with the sets $B'$ and $C$ defines a $S_{a,b}$, where again the color of the vertices in $C$ is different from the color of vertices in $B'$. 
\end{proof}



\section{Acknowledgements}
The research of all three authors was supported by the National Research, Development and Innovation Office -- NKFIH under the grant K116769.


\begin{thebibliography}{1}

\bibitem{erdossosconj}
Mikl{{\'o}}s Ajtai, J{\'a}nos Koml{{\'o}}s, Mikl{\'o}s Simonovits, and Endre
  Szemer{\'e}di.
\newblock Proof of the {E}rd{\H{o}}s-{T.} {S{\'o}}s conjecture for large trees,
  in preparation.

\bibitem{erd1964extremal}
Paul Erd{\H{o}}s.
\newblock Extremal problems in graph theory.
\newblock In {\em {T}heory of graphs and its applications}. Smolenice
  Publishing House of the Czechoslovak Academy of Science, Prague, 1964.

\bibitem{erdHos1959maximal}
Paul Erd{\H{o}}s and Tibor Gallai.
\newblock On maximal paths and circuits of graphs.
\newblock {\em Acta Mathematica Hungarica}, 10(3-4):337--356, 1959.

\bibitem{furedi2015extremal}
Zoltan F{\"u}redi and David Gunderson.
\newblock Extremal numbers for odd cycles.
\newblock {\em Combinatorics, Probability and Computing}, 24(4):641--645, 2015.

\bibitem{gyHori20123}
Ervin Gy{\H{o}}ri and Nathan Lemons.
\newblock 3-uniform hypergraphs avoiding a given odd cycle.
\newblock {\em Combinatorica}, 32(2):187--203, 2012.

\bibitem{mclennan2005erdHos}
Andrew McLennan.
\newblock The {Erd{\H{o}}s-S{\'o}s} conjecture for trees of diameter four.
\newblock {\em Journal of Graph Theory}, 49(4):291--301, 2005.

\bibitem{simonovits1968method}
Mikl{\'o}s Simonovits.
\newblock A method for solving extremal problems in graph theory, stability
  problems.
\newblock In {\em Theory of Graphs (Proc. Colloq., Tihany, 1966)}, pages
  279--319, 1968.

\end{thebibliography}
\end{document}